\documentclass[reqno]{amsart}

\usepackage{amscd, amssymb, amsthm, mathrsfs, amsmath, amsrefs}
\usepackage[all]{xy}

\numberwithin{equation}{section}

\theoremstyle{plain}

\newtheorem{lem}[equation]{Lemma}
\newcommand{\comment}[1]{}

\theoremstyle{definition}
\newtheorem{defi}[equation]{Definition}

\theoremstyle{remark}

\newtheorem{rem}[equation]{Remark}

\newtheorem*{ack}{Acknowledgements}
\theoremstyle{definition}

\numberwithin{equation}{section}

\usepackage{tikz}
\usetikzlibrary{matrix,arrows}

\newcommand{\sch}{\mathbf{Sch}_k}

\newcommand{\dgcat}{\mathbf{dgCat}_k}

\newcommand{\perf}{\mathrm{Perf}}
\newcommand{\Dperf}{\mathbf{D}_{\mathrm{perf}}}
\newcommand{\Perf}{\mathbf{Perf}}
\newcommand{\ac}{\mathrm{Ac}}

\newcommand{\Oo}{\mathcal{O}}

\newcommand{\fS}{\mathfrak{S}}
\newcommand{\fT}{\mathfrak{T}}

\newcommand{\op}{\mathrm{op}}

\newcommand{\Z}{\mathbb{Z}}

\DeclareMathOperator{\Spec}{Spec}

\newcommand{\Op}{\mathrm{Op}}
\newcommand{\scrD}{\mathscr{D}}
\newcommand{\scrC}{\mathscr{C}}
\newcommand{\scrF}{\mathscr{F}}
\newcommand{\Pf}{\mathcal{P}_f}
\newcommand{\im}{\mathrm{im}\hspace{0.1em}}
\newcommand{\fU}{\mathfrak{U}}
\newcommand{\catk}{\mathbf{Cat}_k}
\newcommand{\Mod}{\mathbf{Mod}}

\renewcommand{\L}{\mathbf{L}}

\begin{document}
\title[A dg--enhancement of the derived category of perfect complexes]{A strictly-functorial and small dg--enhancement of the derived category of perfect complexes}
\author{Emanuel Rodr\'\i guez Cirone}
\email{ercirone@dm.uba.ar}
\address{Dep. Matem\'atica-IMAS, FCEyN-UBA\\ Ciudad Universitaria Pab 1\\
1428 Buenos Aires\\ Argentina}
\begin{abstract}
Let $k$ be a commutative noetherian ring. We construct a strictly--functorial presheaf of small dg--categories over $k$ on the category of $k$--schemes of finite type, which gives dg--enhancements of the derived categories of perfect complexes.
\end{abstract}

\maketitle

\section{Introduction}

Let $k$ be a commutative noetherian ring and let $\dgcat$ be the category of small dg--categories over $k$ and dg--functors; see \cite{kellericm} for background on dg--categories. Let $\sch$ be the category of $k$--schemes of finite type. For $\fS\in\sch$ let $\Dperf(\fS)$ denote the derived category of perfect complexes on $\fS$; see \cite{tt}*{Section 2} for the definition of perfect complex and \cite{tt}*{1.9.6} for the derived category. We construct a strictly--functorial presheaf of small dg--categories on $\sch$, $\fS\mapsto\Perf_\fS$, such that $H^0(\Perf_\fS)$ is naturally equivalent to $\Dperf(\fS)$. Moreover, every morphism $f:\fT\to\fS$ induces a functor $H^0(\Perf_\fS)\to H^0(\Perf_\fT)$ that identifies with the left derived inverse image functor $\L f^*:\Dperf(\fS)\to\Dperf(\fT)$.

The existence of a strictly--functorial and small dg--enhancement of $\Dperf(\fS)$ is well--known to experts \cite{chsw}*{Example 2.7} and we make no claim to originality. However, a detailed construction of it seems to be difficult to find in the literature and this article tries to fill in this gap. Our approach is more explicit than the one suggested in \cite{chsw}*{Example 2.7} and it avoids the use of Grothendieck universes to guarantee the smallness of the dg--categories involved.

\section{Construction of the presheaf $\Perf_?:\sch^\op\to\dgcat$}

For simplicity, throughout the word \emph{scheme} means $k$--scheme of finite type.

\subsection{The $\Oo_\fS$--modules $\Oo_{\fS,V}$}\label{subsec:extensionbyzero}

Let $\fS$ be a scheme and let $i_V:V\subseteq\fS$ be the inclusion of an open subset. Let $\Mod(\fS)$ denote the category of $\Oo_\fS$--modules. The inverse image functor $i_V^*:\Mod(\fS)\to\Mod(V)$ is right adjoint to the \emph{extension by zero} functor $(i_V)_!$. Put $\Oo_{\fS,V}:=(i_V)_!\Oo_V$. For any $\Oo_\fS$--module $M$ we have an isomorphism
\[\sigma:\hom_\fS(\Oo_{\fS,V},M)\cong \hom_V(\Oo_V,M|_V)\cong \Gamma(V,M)\]
that is natural in $M$. Now let $f:\fT\to\fS$ be a morphism of schemes and let $N$ be an $\Oo_\fT$--module. The isomorphism
\begin{align*}
\begin{aligned}
\hom_\fT(\Oo_{\fT,f^{-1}V},N) & \overset{\sigma}\cong\Gamma(f^{-1}V,N) \\
& = \Gamma(V,f_*N) \\
& \overset{\sigma}\cong\hom_\fS(\Oo_{\fS,V},f_*N) \\
& \cong\hom_\fT(f^*\Oo_{\fS,V},N)
\end{aligned}
\end{align*}
is natural in $N$, and as such it is induced by a unique isomorphism of $\Oo_\fT$--modules $\theta_{f,V}:f^*\Oo_{\fS,V}\to\Oo_{\fT,f^{-1}V}$.

\subsection{Rectification of the inverse image functor on the modules $\Oo_{\fS,V}$} Any morphism $f:\fT\to\fS$ induces an inverse image functor $f^*:\Mod(\fS)\to\Mod(\fT)$. However, if $g:\fU\to \fT$ is another morphism and $M$ is an $\Oo_\fS$--module, the modules $g^*f^*M$ and $(fg)^*M$ are not the same but only naturally isomorphic; we will need an equality on the nose in order to define the presheaf $\Perf_?$.

Let $\fS$ be a scheme and let $\Op(\fS)$ be the poset of open subsets of $\fS$. Let $\scrF_\fS$ be the full subcategory of $\Mod(\fS)$ whose objects are the modules $\Oo_{\fS,V}$ with $V\in \Op(\fS)$. Notice that $\scrF_{\fS}$ is a small category. For every morphism $f:\fT\to \fS$ we will define a functor $f^\star:\scrF_\fS\to\scrF_\fT$ such that, in the situation of the previous paragraph, $g^\star f^\star M=(fg)^\star M$. For an open subset $V\subseteq\fS$ put $f^\star \Oo_{\fS,V}:= \Oo_{\fT,f^{-1}V}$; this defines $f^\star$ on the objects. The definition of $f^\star$ on the morphisms is given by
\begin{align*}
\begin{aligned}
f^\star:\hom_\fS(\Oo_{\fS,V},\Oo_{\fS,W}) &\to\hom_\fT(\Oo_{\fT,f^{-1}V},\Oo_{\fT,f^{-1}W}) \\
\varphi &\mapsto \theta_{f,W}\circ f^*(\varphi)\circ \theta_{f,V}^{-1}
\end{aligned}
\end{align*}
where the morphisms $\theta$ are the ones defined in \ref{subsec:extensionbyzero}.

\begin{lem}
Let $\catk$ denote the category of small $k$-linear categories. Then the definitions above give rise to a functor $\scrF_?:\sch^\op\to\catk$.
\end{lem}
\begin{proof}
Let $f:\fT\to\fS$ and $g:\fU\to\fT$ be morphisms, let $f^*$ and $g^*$ be the usual inverse image functors on modules, and let $\alpha:(fg)^*\to g^*f^*$ be the natural isomorphism induced by the adjunctions between inverse and direct images of modules. Let $\varphi\in\hom_\fS(\Oo_{\fS,V},\Oo_{\fS,W})$. In order to prove the equality $g^\star(f^\star\varphi)=(fg)^\star\varphi$, it suffices to show that
\begin{equation}\label{eq:relationupperstar}\theta_{g,f^{-1}W}\circ g^*(\theta_{f,W})\circ \alpha_{\Oo_{\fS,W}}=\theta_{fg,W}.\end{equation}
One way to prove \eqref{eq:relationupperstar} is to show that both morphisms induce the same function upon applying the functor $\hom_\fU(?,M)$ for any $\Oo_\fU$--module $M$; this is straightforward from the definitions.
\end{proof}

\begin{rem}
By definition of $f^\star$, the morphisms $\theta_{f,V}$ defined in \ref{subsec:extensionbyzero} assemble into a natural isomorphism between the restriction of $f^*:\Mod(\fS)\to\Mod(\fT)$ to $\scrF_\fS$ and the composite functor $f^\star:\scrF_\fS\to\scrF_\fT\subseteq \Mod(\fT)$.
\end{rem}

\subsection{Resolution by direct sums of $\Oo_{\fS,V}$} It is well known that the $\Oo_{\fS,V}$ are flat $\Oo_\fS$--modules, and that every $\Oo_\fS$--module is a quotient of a direct sum of these with varying $V$ \cite{hart}*{II, Proposition 1.2}. We will use a version of this last fact that keeps track of some cardinality issues. From now on, $\kappa$ is an infinite cardinal such that $\kappa\geq |k|$. We begin with the following easy observation.

\begin{lem}\label{lem:card}
Let $\fS\in\sch$. Then $|\Op(\fS)|\leq \kappa$.
\end{lem}
\begin{proof}
We may assume that $\fS$ is an affine scheme, since it is a finite union of affine open subschemes. Suppose that $\fS=\Spec A$. Notice that $|A|\leq\kappa$ since $A$ is a $k$--algebra of finite type. We have
\[|\Op(\fS)|=|\left\{ D(f):f\in A\right\}|\leq |A|\leq\kappa\]
where the equality on the left follows from the fact that every open subset of $\fS$ is a finite union of distinguished open subsets.
\end{proof}

\begin{lem}\label{lem:keycondition}
Let $J$ be an index set of cardinality $\kappa$. Let $M$ be an $\Oo_\fS$--module such that $|M(U)|\leq \kappa$ for all $U\in\Op(\fS)$. Let $N$ be an $\Oo_\fS$--module and let $\pi:N\twoheadrightarrow M$ be an epimorphism. Then there exists a function $V_?:J\to\Op(\fS)$ and a morphism $\bigoplus_{j\in J}\Oo_{\fS,V_j}\to N$ such that the composite $\bigoplus_{j\in J}\Oo_{\fS,V_j}\to M$ is an epimorphism.
\end{lem}
\begin{proof}
We have a commutative diagram of $\Oo_\fS$--modules
\[\xymatrix{
\bigoplus_{V\in\Op(\fS)}\bigoplus_{t\in N(V)}\Oo_{\fS,V}\ar[r]\ar@{->>}[d]& \bigoplus_{V\in\Op(\fS)}\bigoplus_{s\in M(V)}\Oo_{\fS,V}\ar@{->>}[d] \\
N\ar@{->>}[r]^\pi & M \\ }\]
where the right vertical arrow is induced by the isomorphisms $\hom_\fS(\Oo_{\fS,V},M)\cong M(V)$, and the left vertical arrow is obtained in an similar manner. Notice that the index set of the direct sum on the left may have cardinality greater than $\kappa$. For each $s\in\im(\pi_V:N(V)\to M(V))$, choose $t_s\in N(V)$ such that $\pi_V(t_s)=s$. Put $\tilde{N}(V):=\left\{t_s\in N(V):s\in \im \pi_V\right\}$ and observe that $|\tilde{N}(V)|\leq \kappa$. Let $\rho$ be the obvious morphism $\bigoplus_{V\in\Op(\fS)}\bigoplus_{t\in\tilde{N}(V)}\Oo_{\fS,V}\to N$. It is easily verified that $\pi\circ\rho$ is an epimorphism. Moreover, the source of $\rho$ is a direct sum over the set $S:=\coprod_{V\in\Op(\fS)}\tilde{N}(V)$ and we have $|S|\leq |\Op(\fS)|\cdot \kappa\leq\kappa\cdot\kappa=\kappa$.
\end{proof}

From now on we fix a set $J$ of cardinality $\kappa$.

\subsection{The dg--category $\Perf_\fS$.} Let $\scrD_\fS$ be the full subcategory of $\Mod(\fS)$ whose objects are the modules $\bigoplus_{j\in J}\Oo_{\fS,V_j}$ for all possible functions $V_?:J\to \Op(\fS)$. It is a small category, since its objects are in bijection with the set of functions $J\to \Op(\fS)$. Moreover it is an additive category; indeed, we can form the direct sum of two objects of $\scrD_\fS$ using a bijection $J\cong J\coprod J$. We now proceed to construct the dg--category $\Perf_\fS$; we start by defining a dg--category $\perf_\fS$. The objects of $\perf_\fS$ are the perfect (strictly) bounded above cochain complexes in $\scrD_\fS$. For two of such complexes $E$ and $F$, let $\perf_\fS(E,F)^n$ be the $k$--module of families $\varphi=(\varphi^p)_{p\in \Z}$ of morphisms $\varphi^p:E^p\to E^{p+n}$. Let $\perf_\fS(E,F)$ be the cochain complex of $k$--modules with components $\perf_\fS(E,F)^n$ and differential given by the usual formula:
\[d(\varphi)=d_F\circ \varphi-(-1)^n\varphi\circ d_E.\]
Notice that the dg--category $\perf_\fS$ is small because $\scrD_\fS$ is. Let $\ac_\fS$ be the full dg--subcategory of acyclic complexes of $\perf_\fS$ and define $\Perf_\fS$ to be the Drinfeld quotient $\perf_\fS/\ac_\fS$ \cite{drin}*{3.1}. Notice that this Drinfeld quotient is well defined since $\ac_\fS$ is a small dg-category.

\begin{defi}Let $T$ be a triangulated category. A \emph{dg--enhancement} of $T$ is a strongly pretriangulated dg--category $\mathcal{T}$ \cite{drin}*{2.4} such that $H^0(\mathcal{T})$ is triangle equivalent to $T$.\end{defi}

\begin{lem}
$\Perf_\fS$ is a dg--enhancement of the derived category $\Dperf(\fS)$ of perfect complexes on $\fS$.
\end{lem}
\begin{proof}
First observe that the dg--category $\perf_\fS$ is strongly pretriangulated; this is because the shifts of objects of $\perf_\fS$ and the cones of morphisms of complexes between objects of $\perf_\fS$ are isomorphic to objects of $\perf_\fS$. Note that $\ac_\fS$ is strongly pretriangulated as well, since the shifts of acyclic complexes are acyclic and the cones of morphisms of complexes between acyclic complexes are acyclic too. Hence $\Perf_\fS$ is a strongly pretriangulated dg--category. Moreover, we have equivalences of triangulated categories
\[H^0(\Perf_\fS)\cong H^0(\perf_\fS)/H^0(\ac_\fS)\cong \Dperf(\fS).\]
Here the equivalence on the left follows from \cite{drin}*{Theorem 3.4} and it remains to prove the equivalence on the right. By \cite{tt}*{1.9.7}, to show that $H^0(\perf_\fS)/H^0(\ac_\fS)\cong \Dperf(\fS)$ it suffices to prove the following: for every perfect complex of $\Oo_\fS$--modules $E$ there exists a bounded above complex $F$ of modules in $\scrD_\fS$ and a quasi--isomorphism $F\overset{\sim}\to E$. The existence of such a quasi-isomorphism is proved in Lemma \ref{lem:resolution}, see Remark \ref{rem:perfect}.
\end{proof}

\begin{lem}\label{lem:card}
Let $M$ be an $\Oo_\fS$--module which is either
\begin{enumerate}
\item a coherent $\Oo_\fS$--module, or
\item a subquotient of an object of $\scrD_\fS$.
\end{enumerate}
Then $|M(U)|\leq \kappa$ for all $U\in\Op(\fS)$.
\end{lem}
\begin{proof}
Suppose first that $M$ is a coherent $\Oo_\fS$--module. Since $\fS$ is Noetherian, every open subset of $\fS$ is quasi-compact, and we can write $U$ as the union of a finite number of affine open subsets $\Spec A_i$ ($1\leq i\leq k$). Notice that $A_i$ has cardinality at most $\kappa$ since it is a $k$--algebra of finite type. Notice also that $M(\Spec A_i)$ has cardinality at most $\kappa$ since it is a finite $A_i$--module. It follows that $|M(U)|\leq \kappa$ since we have an injection $M(U)\to \prod_{i=1}^k M(\Spec A_i)$.

Now let $N$ be a presheaf of abelian groups on $\fS$ and let $aN$ be its associated sheaf. Suppose that $|N(U)|\leq \kappa$ for every open set $U$. We will show that $|aN(U)|\leq\kappa$ for every open set $U$. For every $s\in aN(U)$ there exist a finite open cover $U=W_1\cup\cdots\cup W_k$ and sections $s_i\in N(W_i)$ such that $s|_{W_i}=s_i$ for all $i$. This implies that there is a surjection
\[\coprod_{F\in\Pf(\Op(U))}\left[ \lim_{W\in F} N(W)\right]\to aN(U)\]
where $\Pf(?)$ is the set of finite subsets of $?$, and $\lim_{W\in F}N(W)$ is the subset of the product $\prod_{W\in F}N(W)$ formed by those tuples of sections which agree on double intersections. It follows that $|aN(U)|\leq \kappa$ since $|\Pf(\Op(U))|=|\Op(U)|\leq\kappa$.

As a consequence of the previous paragraphs, we see that $|\Oo_{\fS,V}(U)|\leq \kappa$ for all $U,V\in\Op(\fS)$. Indeed, $\Oo_{\fS,V}$ is the sheaf associated to the abelian presheaf
\[U\mapsto \left\{\begin{array}{ll}\Oo_\fS(U) & \mbox{if }U\subseteq V\mbox{,} \\
0 & \mbox{if not.}\\ \end{array}\right.\]

Now let $M$ be an object of $\scrD_\fS$, that is, $M=\bigoplus_{j\in J} \Oo_{\fS,V_j}$. Let $U$ be an open subset of $\fS$. Since elements of a direct sum have finite support, there is a surjection
\[\coprod_{F\in\Pf(J)}\left[ \bigoplus_{j\in F} \Oo_{\fS, V_j}(U)\right]\to M(U).\]
It follows that $|M(U)|\leq\kappa$ since $|\Pf(J)|=|J|=\kappa$.

Suppose now that $M$ is a submodule of an object of $\scrD_\fS$. It is clear from the previous paragraph that $|M(U)|\leq \kappa$ for every open $U\in\Op(\fS)$. Finally, if $N$ is a submodule of such an $M$, it follows that $|(M/N)(U)|\leq\kappa$ for every open $U$, since $M/N$ is the sheaf associated to the presheaf $U\mapsto M(U)/N(U)$.
\end{proof}

We say that a cochain complex of $\Oo_\fS$--modules $E$ is \emph{cohomologically bounded above} if there exists a natural number $N$ such that $H^n(E)=0$ for all $n\geq N$.

\begin{lem}\label{lem:resolution}
Let $E$ be a cohomologically bounded above complex of $\Oo_\fS$--modules such that $|H^n(E)(U)|\leq\kappa$ for all $n$ and all $U\in\Op(\fS)$. (For example, $E$ can be a perfect complex of $\Oo_\fS$--modules, as explained in Remark \ref{rem:perfect}.) Then there exists a bounded above complex of modules in $\scrD_\fS$, say $F$, and a quasi-isomorphism $F\overset{\sim}\to E$.
\end{lem}
\begin{proof}
It will follow directly from \cite{tt}*{Lemma 1.9.5}. Let $\scrC$ be the full subcategory of the category of cochain complexes of $\Oo_\fS$--modules whose objects are those cohomologically bounded above complexes $E$ such that $|H^n(E)(U)|\leq \kappa$ for all $n$ and all $U\in\Op(\fS)$. Notice that any complex of $\Oo_\fS$--modules quasi-isomorphic to a complex in $\scrC$ is itself in $\scrC$. Notice also that any strict bounded complex in $\scrD_\fS$ is in $\scrC$ by Lemma \ref{lem:card} (2). It is easily verified that $\scrC$ contains the mapping cone of any map of complexes $F\to E$ with $E\in \scrC$ and $F$ a strict bounded complex in $\scrD_\fS$. The \emph{key condition} \cite{tt}*{1.9.5.1} holds by Lemma \ref{lem:keycondition}. The statement now follows from \cite{tt}*{Lemma 1.9.5}.
\end{proof}

\begin{rem}\label{rem:perfect}
In the statement of Lemma \ref{lem:resolution} we can take $E$ to be any perfect complex of $\Oo_\fS$--modules. Indeed, any perfect complex has coherent cohomology groups since $\fS$ is Noetherian, and so satisfies the hypothesis of Lemma \ref{lem:resolution} by Lemma \ref{lem:card} (1).
\end{rem}

\subsection{Functoriality of $\Perf_?$} Let $f:\fT\to\fS$ be a morphism in $\sch$. The functor $f^\star:\scrF_\fS\to\scrF_\fT$ extends to a $k$--linear functor $f^\star:\scrD_\fS\to\scrD_\fT$ in a natural way, as we proceed to explain. To define $f^\star$ on the objects of $\scrD_\fS$ put
\[ f^\star\left(\bigoplus_{j\in J}\Oo_{\fS,V_j}\right):=\bigoplus_{j\in J}f^\star\Oo_{\fS,V_j}.\]
Now observe that the functors $\hom_\fS(\Oo_{\fS,V},?)$ commute with direct sums and define $f^\star$ on morphisms by the dashed arrow that makes the following diagram commute.
\[\xymatrix{
\hom_\fS\left(\bigoplus_j\Oo_{\fS,V_j},\bigoplus_k\Oo_{\fS,W_k}\right)\ar[r]^\cong\ar@{-->}[d]& \prod_j\bigoplus_k\hom_\fS(\Oo_{\fS,V_j},\Oo_{\fS, W_k})\ar[d]^{\prod_j\bigoplus_k f^\star} \\
\hom_\fT\left(\bigoplus_jf^\star\Oo_{\fS,V_j},\bigoplus_kf^\star\Oo_{\fS,W_k}\right)\ar[r]^\cong & \prod_j\bigoplus_k\hom_\fT(f^\star\Oo_{\fS,V_j},f^\star\Oo_{\fS, W_k}) \\ }\]
Now that we have defined a $k$--linear functor $f^\star:\scrD_\fS\to\scrD_\fT$, we apply it degreewise to obtain a dg--functor $f^\star:\perf_\fS\to\perf_\fT$. Finally, this dg--functor takes objects of $\ac_\fS$ to objects of $\ac_\fT$ \cite{lipman}*{Proposition 2.7.2, Example 2.7.3}, and so it induces a dg--functor $f^\star:\Perf_\fS\to\Perf_\fT$. It is straightforward from the above definition that if $f$ and $g$ are two composable morphisms in $\sch$, then the dg--functor $g^\star\circ f^\star$ equals $(fg)^\star$. Hence, we have constructed a presheaf of dg--categories $\Perf_?:\sch^\op\to\dgcat$.

\begin{rem}
It is straightforward to see that the isomorphisms
\[f^*\bigoplus_j\Oo_{\fS,V_j}\cong \bigoplus_jf^*\Oo_{\fS,V_j}\overset{\bigoplus \theta_{f,V_j}}\longrightarrow \bigoplus_j\Oo_{\fT,f^{-1}V_j}=f^\star\bigoplus_j\Oo_{\fS,V_j}\]
assemble into a natural isomorphism between the restriction of $f^*:\Mod(\fS)\to\Mod(\fT)$ to $\scrD_\fS$ and the composite functor $f^\star:\scrD_\fS\to\scrD_\fT\subseteq \Mod(\fT)$. Hence the following diagram commutes up to natural isomorphism of functors.
\[\xymatrix{
H^0(\Perf_\fS)\ar[r]^\cong\ar[d]_{H^0(f^\star)}& \Dperf(\fS)\ar[d]^{\L f^*} \\
H^0(\Perf_\fT)\ar[r]^\cong & \Dperf(\fT) \\ }\]
\end{rem}

\begin{ack}
The author wishes to thank Guillermo Corti\~nas for his useful suggestions regarding these notes.
\end{ack}


\begin{bibdiv}
\begin{biblist}

\bib{chsw}{article}{
   author={Corti{\~n}as, Guillermo},
   author={Haesemeyer, Christian},
   author={Schlichting, Marco},
   author={Weibel, Charles},
   title={Cyclic homology, cdh-cohomology and negative $K$-theory},
   journal={Ann. of Math. (2)},
   volume={167},
   date={2008},
   number={2},
   pages={549--573},
}

\bib{drin}{article}{
   author={Drinfeld, Vladimir},
   title={DG quotients of DG categories},
   journal={J. Algebra},
   volume={272},
   date={2004},
   number={2},
   pages={643--691},
}


\bib{hart}{book}{
   author={Hartshorne, Robin},
   title={Residues and duality},
   series={Lecture notes of a seminar on the work of A. Grothendieck, given
   at Harvard 1963/64. With an appendix by P. Deligne. Lecture Notes in
   Mathematics, No. 20},
   publisher={Springer-Verlag, Berlin-New York},
   date={1966},
   pages={vii+423},
   review={\MR{0222093 (36 \#5145)}},
}


\bib{kellericm}{article}{
   author={Keller, Bernhard},
   title={On differential graded categories},
   conference={
      title={International Congress of Mathematicians. Vol. II},
   },
   book={
      publisher={Eur. Math. Soc., Z\"urich},
   },
   date={2006},
   pages={151--190},
}


\bib{lipman}{article}{
   author={Lipman, Joseph},
   title={Notes on derived functors and Grothendieck duality},
   conference={
      title={Foundations of Grothendieck duality for diagrams of schemes},
   },
   book={
      series={Lecture Notes in Math.},
      volume={1960},
      publisher={Springer, Berlin},
   },
   date={2009},
   pages={1--259},
   review={\MR{2490557 (2011d:14029)}},
   doi={10.1007/978-3-540-85420-3},
}

\bib{tt}{article}{
   author={Thomason, R. W.},
   author={Trobaugh, Thomas},
   title={Higher algebraic $K$-theory of schemes and of derived categories},
   conference={
      title={The Grothendieck Festschrift, Vol.\ III},
   },
   book={
      series={Progr. Math.},
      volume={88},
      publisher={Birkh\"auser Boston},
      place={Boston, MA},
   },
   date={1990},
   pages={247--435},
   review={\MR{1106918 (92f:19001)}},
}



\end{biblist}
\end{bibdiv}
\end{document}